\documentclass[11pt,headsepline, a4paper, twoside,cleardoubleempty]{scrartcl}

\usepackage[utf8]{inputenc}
\usepackage{geometry}
\usepackage[intlimits]{amsmath}
\usepackage{amsfonts}
\usepackage{amssymb}
\usepackage[amsmath, thmmarks]{ntheorem}
\usepackage{enumerate}
\usepackage[arrow, matrix, curve]{xy}
\usepackage{mathtools}
\usepackage{tikz}
\usepackage{float}
\usepackage{scrpage2}
\usepackage[nottoc]{tocbibind}
\usepackage[pdftoolbar]{hyperref}
\usepackage{pdflscape}
\usepackage{listings}  
\hypersetup{colorlinks=true, linkcolor=blue, citecolor=blue}
\usepackage{lipsum}

\newcommand{\R}{\mathbb{R}}
\newcommand{\Z}{\mathbb{Z}}
\newcommand{\Q}{\mathbb{Q}}
\newcommand{\N}{\mathbb{N}}
\newcommand{\C}{\mathbb{C}}

\renewcommand{\sl}{\operatorname{SL}(2,\Z)}

\makeatother
\newtheorem{vorlage}{}[section]
\newtheorem{prop}[vorlage]{Proposition}
\newtheorem{lemma}[vorlage]{Lemma}
\newtheorem{cor}[vorlage]{Corollary}

\newtheorem{theorem}[vorlage]{Theorem}
\theorembodyfont{\rmfamily}
\newtheorem{definition}[vorlage]{Definition}
\newtheorem{rem}[vorlage]{Remark}

\theoremstyle{nonumberbreak}
\theoremstyle{nonumberplain}
\theoremsymbol{\ensuremath{\square}}
\newtheorem{proof}{Proof}

\clearscrheadfoot
\rehead{\headmark}
\lehead{\pagemark}
\lohead{\headmark}
\rohead{\pagemark} 

\title{The graded ring of modular forms on the Cayley half-space of degree two}
\author{C. Dieckmann, A. Krieg, M. Woitalla}

\begin{document}
\maketitle
\begin{abstract}
A result by Hashimoto and Ueda says that the graded ring of modular forms with respect to $\operatorname{SO}(2,10)$ is a polynomial ring in modular forms of weights $4, 10, 12, 16, 18, 22, 24, 28, 30, 36, 42$. In this paper we show that one may choose Eisenstein series as generators. This is done by calculating sufficiently many Fourier coefficients of the restrictions to the Hermitian half-space. Moreover we give two constructions of the skew symmetric modular form of weight 252.
\end{abstract}
\vspace{3mm}

\noindent\textbf{MSC-Classification:} 11F55

\section{Introduction}
\noindent Let $V$ be a real quadratic space of signature $(2,10)$. The bilinear form of $V$ is denoted by $(\cdot,\cdot)$. The group of all isometries of $V$ is called the \textit{orthogonal group of} $V$ and is given by
\[\operatorname{O}(V)=\{g\in\operatorname{GL}(V)\,|\,(gv,gv)=(v,v)\;\;\text{for all}\;\; v\in V\}\,.\]
By $\operatorname{SO}(V)$ we denote the subgroup of index two, which is equal to the kernel of the determinant-character. We obtain another subgroup $\operatorname{O}(V)^{+}$ of index two  as the kernel of the real spinor norm. The intersection of the groups $\operatorname{SO}(V)$ and $\operatorname{O}(V)^{+}$ is denoted by $\operatorname{SO}(V)^{+}$. This is the connected component of the identity and is well-known to be a  semisimple and noncompact Lie group, compare e.g. \cite{Kp}. Its maximal compact subgroup is given by $\operatorname{SO}(2)\times\operatorname{SO}(10)$. In this text we fix $L=E_{8}$ to be the (up to isometries) unique even positive definite unimodular lattice in dimension 8. 
\noindent We denote the Gram matrix of $L$ by $S$ and define the even unimodular lattice of signature $(2,10)$ by
\[\operatorname{Gram}(L_{2})=\begin{pmatrix}
					  0&0&0&0&1\\
					  0&0&0&1&0\\   
					  0&0&-S&0&0\\
					  0&1&0&0&0\\
					  1&0&0&0&0
                                         \end{pmatrix}\in\operatorname{GL}(L_{2})\subseteq \operatorname{GL}(V)\,.\]
We consider the arithmetic subgroup 
 \[\operatorname{O}(L_{2})^{+}=\{g\in\operatorname{O}(V)^{+}\,|\,g\,L_{2}\subseteq L_{2}\}\] 
and define the sublattice $L_{1}$ by
\[S_{1}:=\operatorname{Gram}(L_{1})=\begin{pmatrix}
					  0&0&1\\   
					  0&-S&0\\
					  1&0&0\\
                                         \end{pmatrix}\in\operatorname{GL}(L_{2})\subseteq \operatorname{GL}(V)\,.\]
The following Proposition is well-known, confer \cite{GHS}, Theorem 1.7, and \cite{KW}, Corollary 2. 
\begin{prop}
 The commutator subgroup of $\operatorname{SO}(L_{2})^{+}$ is trivial. The group of characters of $\operatorname{O}(L_{2})^{+}$ is a group of order two, which is generated by the determinant-character.
\end{prop}
We extend the bilinear form of $V$ to $V\otimes\C$ by $\C$-linearity. The group $\operatorname{O}(V)^{+}$ acts on the domain
\[\mathcal{D}=\{[\mathcal{Z}]\in\mathbb{P}(V\otimes\C)\,|\,(\mathcal{Z},\mathcal{Z})=0\,,\,(\mathcal{Z},\overline{\mathcal{Z}})>0\}^{+}\] 
as a linear group. The superscript means that we have chosen one of the two connected components. There is an equivalent affine model for this domain given by
\begin{equation}\label{affine model for hom domain}
\mathcal{H}(L_{2})=
\left\{\begin{pmatrix}
         \omega\\z\\\tau
        \end{pmatrix}
\in \C\times (L\otimes \C)\times \C\,\left|\begin{aligned}
&\omega_{i},\tau_{i}>0\,,&\\&2\omega_{i}\tau_{i}-(z_{i},z_{i})>0
\end{aligned}\right.\right\}\,,
\end{equation}
where we have used the abbreviations
\[\omega_{i}:=\operatorname{Im}(\omega),\quad \tau_{i}:=\operatorname{Im}(\tau),\quad z_{i}:=\operatorname{Im}(z)\,.\]
The real orthogonal group acts on this domain according to
\[Z\mapsto M\langle Z\rangle =M\{Z\}^{-1}\cdot(-\frac{1}{2}(Z,Z)b+KZ+c)\,,\]
where
\[\begin{aligned}
   &Z\in \mathcal{H}(L_{2})\,,&&M=\begin{pmatrix}\alpha & a& \beta \\ b & K & c \\ \gamma & d & \delta \end{pmatrix}\in \operatorname{O}(V)^{+}\,,\,K\in\operatorname{Mat}(10,\R)&\\
   &&&M\{Z\}:=-\frac{\gamma}{2}(Z,Z)+dZ+\delta\,.&
  \end{aligned}\]
We introduce the Cayley numbers.
\begin{definition}
 The \textit{Cayley numbers} $\mathcal{C}$ are defined as an 8-dimensional algebra over $\R$ with basis $e_{0},\dots, e_{7}$. They satisfy the following multiplication rules:
\begin{enumerate}[(i)]
 \item $xe_{0}=x=e_{0}x$ for all $x\in \mathcal{C}$,
 \item $e_{j}^{2}=-e_{0}$ for $j=1,\dots,7$,
 \item $e_{1}e_{2}e_{4} = e_{2} e_{3} e_{5} = e_{3} e_{4} e_{6} = e_{4} e_{5} e_{7} = e_{5} e_{6} e_{1} = e_{6} e_{7} e_{2} = e_{7} e_{1} e_{3} = -e_{0} .$
\end{enumerate}
For any $x\in\mathcal{C}$ we write $x=\sum_{j=0}^{7}{x_{j}e_{j}}$. We define the real part as $\operatorname{Re}(x):=x_{0}$ and embed $\R$ into $\mathcal{C}$ by $\R e_{0}$. The involution on $\mathcal{C}$ is the map
\[\mathcal{C}\to\mathcal{C}\,,\,x\mapsto \overline{x}=2x_{0}-x\,,\]
and the norm of $x$ is given by
\[N(x):=\overline{x}x=x\overline{x}=\sum_{j=0}^{7}{x_{j}^{2}}\,.\]
\end{definition}
A matrix $H\in\operatorname{Mat}(2,\mathcal{C})$ is called \textit{Hermitian}, if it has the shape
\[H=\begin{pmatrix}
     a&b\\\overline{b}& d
    \end{pmatrix},\quad\text{ where }\quad a,d\in\R,b\in\mathcal{C}\,.\]
In this case the determinant of $H$ can be defined as 
\[\det(H):=ad-N(b)\,.\]
For $A,B\in\operatorname{Mat}(m,n,\mathcal{C})$ we define the \textit{trace form} by
\[\mathcal{T}(A,B):=\frac{1}{2}\operatorname{trace}(A\overline{B}^{tr}+B\overline{A}^{tr})\,.\] 
A Hermitian matrix $H\in\operatorname{Her}(2,\mathcal{C})$ is called \textit{positive definite}, if 
\[\mathcal{T}(H,g\overline{g}^{tr})>0\quad\text{ for all }0\neq g\in\mathcal{C}^{2}\,.\]
This is equivalent to $a>0$ and $\det H>0$. In this case we write $H>0$.  Let $\mathcal{C}_{\C}:=\mathcal{C}\otimes_{\R}\C$ be the complexification of $\mathcal{C}$. The associated half-plane is the set
\[\mathbb{H}_{2}(\mathcal{C}):=\{Z=X+iY\in \operatorname{Mat}(2,\mathcal{C}_{\C})\,|\,X,Y\in\operatorname{Her}(2,\mathcal{C})\,,\,Y>0\}\,.\]
The half-plane $\mathbb{H}_{2}(\mathcal{C})$ is biholomorphically equivalent to $\mathcal{H}(L_{2})$ and $\operatorname{O}(V)^{+}$ acts on $\mathbb{H}_{2}(\mathcal{C})$ as a group of biholomorphic automorphisms.

\section{The modular group for the even unimodular lattice of signature (\text{$\mathbf{2,10}$)}}
\noindent We define the algebra of \textit{integral Cayley numbers} $\mathcal{O}$ as the $\Z$-module with basis 
\[e_{0},e_{1},e_{2},e_{4},\frac{e_{1}+e_{2}+e_{3}-e_{4}}{2},\frac{e_{5}-e_{0}-e_{1}-e_{4}}{2},\frac{e_{6}-e_{0}+e_{1}-e_{2}}{2},\frac{e_{2}-e_{0}+e_{4}-e_{7}}{2}\,.\]
These elements form the so called \textit{Coxeter-basis}. We turn $\mathcal{O}$ into a positive definite even lattice with quadratic form $N(\cdot)$. The corresponding bilinear form is denoted by $\sigma(a,b)=\overline{a}b+\overline{b}a$. Any $z\in \mathcal{C}_{\C}$ has a unique decomposition as $z=z_{r}+iz_{i}$, where $z_{r},z_{i}\in \mathcal{C}$. As a vector space over the real numbers we have $\mathcal{C}\cong L\otimes\R$. The following biholomorphic map enables us to identify $\mathbb{H}_{2}(\mathcal{C})$ with $\mathcal{H}(L_{2})$
\[\Psi\,:\,\mathcal{H}(L_{2})\to\mathbb{H}_{2}(\mathcal{C})\,,\,(\omega,z,\tau)^{tr}\mapsto Z:=\begin{pmatrix}
                                                                            \omega&z_{r}+iz_{i}\\
									      \overline{z_{r}}+i\overline{z_{i}}&\tau
                                                                           \end{pmatrix}\,.\]
By virtue of the identity 
\begin{equation}\label{determinant-form identity}
 \omega\tau-N(z)=\det(\Psi(\omega,z,\tau))
\end{equation}
 for each $(\omega,z,\tau)\in\mathcal{H}(L_{2})$, we are able to establish the following Lemma. 
\begin{lemma}\label{lemma:biholomorphic correspondence of modular groups}
 We have a homomorphism of groups 
\[\operatorname{O}(V)^{+}\to\operatorname{Bih}(\mathbb{H}_{2}(\mathcal{C}))\]
with kernel $\{\pm I\}$, where the image of $M\in \operatorname{O}(V)^{+}$ is defined according to the following commuting diagram
 \[\begin{xy}
  \xymatrix{
 \mathcal{H}(L_{2})  \ar[r]^{M}   & \mathcal{H}(L_{2})\ar[d]^{\Psi}  \\
    \mathbb{H}_{2}(\mathcal{C}) \ar[u]_{\Psi^{-1}} \ar[r]   & \mathbb{H}_{2}(\mathcal{C})\quad, }
\end{xy}\quad\begin{xy}
  \xymatrix{
 \psi^{-1}(Z)  \ar[r]   & M\langle\psi^{-1}(Z)\rangle\ar[d]  \\
    Z \ar[u] \ar[r]   & \Psi(M\langle\psi^{-1}(Z)\rangle) 
  }
\end{xy}\,.\]
The image of $\operatorname{SO}(L_{2})^{+}$ is generated by the following transformations
\[Z\mapsto Z+H\,,\,H\in \operatorname{Her}(2,\mathcal{O}),\quad\text{ and }\quad Z\mapsto -Z^{-1}\,.\]
We denote this group by $\Gamma_{2}$. The extension of $\Gamma_{2}$ by $Z\mapsto Z^{tr}$ is the image of $\operatorname{O}(L_{2})^{+}$.
\end{lemma}
\begin{proof}
The first statement is clear. We determine the image of $\operatorname{SO}(L_{2})^{+}$. We first note that $\mathcal{O}$ equipped with the quadratic form $N(\cdot)$ is isomorphic to $L=E_{8}$. We fix the Coxeter basis of $\mathcal{O}$. From \cite{GHS}, Theorem 3.4, we obtain the following generators of $\operatorname{SO}(L_{2})^{+}$ 
\[
\begin{aligned}
&T_{x}=\begin{pmatrix}1& -x^{tr}S_{1}&-(x,x)\\0& I_{10}&x\\0&0&1\end{pmatrix}&&\,,\,x=(n,t,m)^{tr}\in L_{1}\text{ and }t\in \mathcal{O}\,, \\\\
   &J=\begin{pmatrix}
      0&0&0&0&-1\\
      0&0&0&-1&0\\
      0&0&I_{8}&0&0\\
      0&-1&0&0&0\\
      -1&0&0&0&0\\
     \end{pmatrix}&&\,.&
  \end{aligned}\]
If we put
\[H:=\Psi(x)=\begin{pmatrix}
              n&t\\\overline{t}&m
             \end{pmatrix}\in\operatorname{Her}(2,\mathcal{O})\] 
then one obtains the transformation $Z\mapsto Z+H$ as the image of $T_{x}$. Moreover (\ref{determinant-form identity}) shows that $Z\mapsto -Z^{-1}$ is the image of $J$. The matrix
\[K:=\begin{pmatrix}
  1 & 0 &  0 & 0 & 0 &-1 &-1& -1\\
 0 &-1&  0 & 0&  0&  0&  0 & 0\\
 0&  0& -1 & 0 & 0 & 0 & 0 & 0\\
 0 & 0 & 0 & -1&  0 & 0 & 0 & 0\\
 0 & 0 & 0 & 0& -1 & 0 & 0 & 0\\
 0 & 0 & 0 & 0 & 0 &-1 & 0 & 0\\
 0 & 0 & 0 & 0 & 0 & 0& -1&  0\\
 0 & 0 & 0 & 0 & 0 & 0 & 0 &-1
  \end{pmatrix}
\]
belongs to the finite orthogonal group $\operatorname{O}(\mathcal{O})$ and 
\[R_{K}=\begin{pmatrix}
         I_{2}&0&0\\0&K&0\\0&0&I_{2}
        \end{pmatrix}\]
is an element of $\operatorname{O}(L_{2})^{+}$ with determinant $-1$, whose image is $Z\mapsto Z^{tr}$. 
\end{proof}
\begin{rem}
 In \cite{EK2} the extension of $\Gamma_{2}$ by  
\[Z\mapsto\overline{U}^{tr}ZU\,,\,U\in \left\{\begin{pmatrix}
                                                                                 0&1\\-1&0
                                                                                \end{pmatrix},\begin{pmatrix}
                                                                                 1&u\\0&1
                                                                                \end{pmatrix}\,,\,u\in\mathcal{O}\right\}\]
is considered instead. In \cite{D2}, pp. 57-63,  the author shows that this group coincides with $\Gamma_{2}$. 
\end{rem}
The affine cone associated to $\mathcal{D}$ is defined as
\[\mathcal{D}^{\bullet}=\{\mathcal{Z}\in V\otimes\C\,|\,[\mathcal{Z}]\in \mathcal{D}\}\,.\]
We introduce the notion of a modular form.
\begin{definition}\label{def:Modular form}
 Let $\Gamma$ be a subgroup of $\operatorname{O}(L_{2})^{+}$ of finite index. A \textit{modular form}\index{modular form} of weight $k\in \Z$ and character $\chi\,:\,\Gamma\to\C^{\times}$ with respect to $\Gamma$ is a holomorphic function $f\,:\,\mathcal{D}^{\bullet}\to\C$ such that 
\[\begin{aligned}
   &f(t\mathcal{Z})=t^{-k}f(\mathcal{Z})\quad\text{ for all }t\in\C^{\times}\,,&\\
   &f(g\mathcal{Z})=\chi(g)f(\mathcal{Z})\quad\text{ for all }g\in\Gamma\,.&
  \end{aligned}
\]
A modular form is called a \textit{cusp form}\index{modular form!cusp form}, if it vanishes at every cusp. The space of modular forms of weight $k$ and character $\chi$ for the group $\Gamma$ will be denoted by $\mathcal{M}_{k}(\Gamma,\chi)$\index{$\mathcal{M}_{k}(\Gamma,\chi)$}. For the subspace of cusp forms we will write $\mathcal{S}_{k}(\Gamma,\chi)$\index{$\mathcal{S}_{k}(\Gamma,\chi)$}.  
\end{definition}
In view of Lemma \ref{lemma:biholomorphic correspondence of modular groups} we can give the corresponding transformation behaviour of a modular form on the domain $\mathbb{H}_{2}(\mathcal{C})$. A holomorphic function $f:\mathbb{H}_{2}(\mathcal{C})\to\C$ is a modular form of weight $k$ with respect to $\Gamma_{2}$ if the following conditions are satisfied:
\begin{enumerate}[(i)]
 \item $f(Z+H)=f(Z)$ for all $H\in\operatorname{Her}(2,\mathcal{O}),Z\in\mathbb{H}_{2}(\mathcal{C})\,,$
 \item $f(-Z^{-1})=\det(Z)^{k}f(Z)$ for all $Z\in\mathbb{H}_{2}(\mathcal{C})\,.$
\end{enumerate}
A modular form $f:\mathbb{H}_{2}(\mathcal{C})\to\C$ is called \textit{symmetric} if $f(Z^{tr})=f(Z)$ and \textit{skew-symmetric} if $f(Z^{tr})=-f(Z)$.

\section{Constructions of modular forms with trivial character}
In \cite{EK} and \cite{EK2} the authors investigate the  Maaß space for the Cayley numbers. Moreover the authors consider Jacobi Eisenstein series and derive an explicit Fourier expansion. In this section we keep up these constructions and describe the Fourier expansion of the Maaß lift. 
Let \(f\) be a modular form of even weight \(k\) with respect to \(\Gamma_2\) with Fourier expansion
\[f(Z)=\sum_{\begin{smallmatrix} 
T\in \operatorname{Her}(2,\mathcal{O}),T\geq 0 \end{smallmatrix}} 
          \alpha \left( T \right)  e^{2 \pi i \mathcal{T} \left( Z,T \right)}\,,\]
where $T=\left(\begin{smallmatrix}
         n & t\\\overline{t} & m
         \end{smallmatrix}\right)\geq 0$ means $n,m\in\N_{0}$ and $\det(T)\geq 0$, and we define 
\[ \varepsilon(T)=\textnormal{max} \left\{ l \in \N; \ \tfrac{1}{l} T \in \operatorname{Her}(2,\mathcal{O}) \right\}\quad\text{ for }T\neq 0\,.\]
 We say that $f$ satisfies the \textit{Maa{\ss} condition} if for all \(0\neq T \in \operatorname{Her}(2,\mathcal{O})\), \(T \geq 0\)
\[\alpha(T) = \sum_{0 < d \mid \varepsilon(T)} d^{k-1} \alpha^{\ast}((\det T)/d^2)) , 
\quad \alpha^{\ast}(n) := \alpha\left(  \left( \begin{smallmatrix} n & 0 \\ 0 & 1 \end{smallmatrix} \right) \right)\,.\]
The space of all such $f$ is called the $\textit{Maa{\ss} space}$ and is denoted by $\mathcal{M}(k,\mathcal{O})$. 
 From \cite{Ei} and \cite{G} we get the existence of an isomorphism between the Maa{\ss} space and the space of elliptic modular forms of weight $k-4$ with respect to $\sl$
 \begin{equation}\label{isom Maas space Jacobi forms}
  \mathcal{M}(k,\mathcal{O}) \rightarrow \mathcal{M}_{k-4}(\sl)\quad, \quad f(Z)\mapsto f^{\ast}(\omega)=\sum_{n=0}^{\infty}{\alpha^{\ast}(n)e^{2\pi i n\omega}}\,.
 \end{equation}
 whenever $k\geq 4$. In the next Theorem we determine the preimages of the normalized elliptic Eisenstein series $g_{k}$, where $g_{0}:=1$. 
\begin{theorem}\label{theorem:Fourier expansion of Maass lifts of JE series}
 For any even $k\geq 10$ the function given by
\[\begin{aligned}
   E_{k}(Z) =& - \tfrac{B_k}{2k} + \sum_{\begin{smallmatrix} T=\left( \begin{smallmatrix} n & t \\ \overline{t} & m \end{smallmatrix} \right) \geq 0 \\ \textnormal{rank}(T)=1 \end{smallmatrix}} 
              \sum_{0<a\mid \varepsilon(T)}a^{k-1} \ e^{2 \pi i \mathcal{T}(T,Z)}& \\
         & - \tfrac{2(k-4)}{B_{k-4}} \sum_{T=\left( \begin{smallmatrix} n & t \\ \overline{t} & m \end{smallmatrix} \right) > 0}
              \sum_{\begin{smallmatrix} 0 < a \mid \varepsilon(T) \end{smallmatrix}} a^{k-1} \ \sigma_{k-5}\left(\tfrac{nm-N(t)}{a^2}\right) \ e^{2 \pi i \mathcal{T}(T,Z)}\quad,\quad Z\in\mathbb{H}_{2}(\mathcal{C})\,,&
  \end{aligned}\]
 defines a symmetric modular form of weight $k$ on $\mathbb{H}_{2}(\mathcal{C})$. If $k=4$ let 
\[ E_{4}(Z) =  \tfrac{1}{240} + \sum_{\begin{smallmatrix} T=\left( \begin{smallmatrix} n & t \\ \overline{t} & m \end{smallmatrix} \right) \geq 0 \\ \textnormal{rank}(T)=1 \end{smallmatrix}} 
              \sum_{0<a\mid \varepsilon(T)}a^{3} \ e^{2 \pi i \mathcal{T}(T,Z)}\,.\]
\end{theorem}
\begin{proof}
 Compare \cite{Ei}.
\end{proof}
In the case $k=4$ we have another description due to \cite{EK2} as a singular modular form
\[240 \,E_{4}(Z)=\sum_{g\in \mathcal{O}^{2}}{e^{2\pi i\mathcal{T}(Z,g\overline{g}^{tr})}}\,.\]
We investigate the restrictions of the $E_{k}$ to the Hermitian half-space.  For this purpose we consider $\mathcal{O}$ with respect to the standard basis of $\mathcal{C}$ and write $t\in\mathcal{O}$ as
\[t=\sum_{j=0}^{7}{t_{j}e_{j}}\quad,\quad t_{j}\in\Z/2\quad,\quad t^{\ast}:=t_{0}e_{0}+t_{1}e_{1}\,.\]
 The lattice
\[\mathfrak{0}^{\sharp}:=\{t_{0}e_{0}+t_{1}e_{1}\,|\,t_{0},t_{1}\in\Z/2\}\subset\mathcal{C}\]
is dual to the ring of integers $\mathfrak{0}$ for the Gaussian number field $\Q(\sqrt{-1})$ where $\sqrt{-1}:=e_{1}$. We consider the embedded Hermitian half-space
\[\mathbb{H}_{2}(\C)=\left\{\left(\begin{smallmatrix}
                             \omega & z \\ \ast & \tau\end{smallmatrix}\right)\in \mathbb{H}_{2}(\mathcal{C});z=z_{0}e_{0}+z_{1}e_{1}
                            \right\}\subseteq \mathbb{H}_{2}(\mathcal{C})\,.\] The Hermitian Maa{\ss} space $\mathcal{M}(k,\mathfrak{o})$ with respect to $\mathfrak{o}$ and $k\in\N_{0}$ was first considered in \cite{Ko}. 
\begin{lemma}\label{lemma:Restriction of Maass space}
 Let $k\geq 4$ be even. The map
\[\mathcal{M}(k,\mathcal{O})\to\mathcal{M}(k,\mathfrak{o})\,,\,f\mapsto f|_{\mathbb{H}_{2}(\C)}\]
 is a homomorphism of vector spaces. The Fourier coefficients of the restrictions are given by
\[\beta \left( \left( \begin{smallmatrix} n & s \\ \overline{s} & m \end{smallmatrix} \right) \right)=\sum_{t\in\mathcal{O},t^{\ast}=s}{\alpha \left( \left( \begin{smallmatrix} n & t \\ \overline{t} & m \end{smallmatrix} \right) \right)}\quad,\quad s\in \mathfrak{o}^{\sharp}\,.\]
\end{lemma}
\begin{proof}
 The resctriction of $f\in\mathcal{M}(k,\mathcal{O})$ satisfies the Maa{\ss} condition for Hermitian modular forms, because Jacobi forms of index 1 are mapped to Jacobi forms of index 1. Hence the map is  well-defined. The remaining part is straight forward.
\end{proof}
For the Maa{\ss} space one only needs the coefficients
\[\beta \left( \left( \begin{smallmatrix} n & s \\ \overline{s} & 1 \end{smallmatrix} \right) \right)=\sum_{r=0}^{n}{\gamma(n-r,s)\alpha^{\ast}(n-r)}\quad,\quad s=0,\frac{e_{0}}{2},\frac{e_{1}}{2},\frac{e_{0}+e_{1}}{2}\]
where 
\[\gamma(n,s):=\sharp\{t\in \mathcal{O}\,|\,t^{\ast}=s\,,\,N(t)=n\}\,.\] 

\section{The graded ring of modular forms}
In this section we construct generators for the ring of modular forms. The starting point of our considerations is a result by Shiga. In this section let  $\Gamma=\operatorname{O}(L_{2})^{+}$. We define the associated modular variety $M:=\mathcal{D}/\Gamma$  and denote by $M^{\ast}$ the Satake-Baily-Borel compactification of $M$. By $\mathbb{P}(\Omega)$ we denote the weighted projective space with weight
\[\Omega=\{4,10,12,16,18,22,24,28,30,36,42\}\,.\]
The following result can be found in \cite{HU}. It was originally published in \cite{L} (cf. also \cite{Shi}). 
\begin{theorem}\label{theorem:Shiga}
 There is an isomorphism $\mathbb{P}(\Omega)\xrightarrow[]{\sim}M^{\ast}$ of algebraic varieties.
\end{theorem}
The graded ring of modular forms with respect to $\Gamma_{2}$ is denoted by $\mathcal{A}$. 
In Corollary 1.3 in \cite{HU} the authors conclude that there exist modular forms $f_k$, $k\in\Omega$ and $k=252$, of weight $k$, which generate $\mathcal{A}$. Observe that $\mathcal{A}$ contains the ring of symmetric modular forms $\mathcal{A}^{\textit{sym}}$ as a subring. 
In Corollary 4.4 in \cite{HU} Hashimoto and Ueda prove an isomorphism of the orbifolds, and identify the line bundles $O(1)$ on them. $\mathcal{A}^{\textit{sym}}$ is just the graded ring obtained as the direct sum over the set of multiples of $O(1)$ of the spaces of sections on $\mathbb{M}\,\backslash\, S_{\mathbb{M}}$, which is the polynomial ring on $\mathbb{T}\,\backslash\, S_{\mathbb{T}}$ side. Thus we have

\begin{cor}\label{cor:abstract ring of symmetric mf}
The ring $\mathcal{A}^{\textit{sym}}$ is a polynomial ring in $11$ variables. More precisely there exist modular forms $f_{k}\in\mathcal{M}_{k}(\Gamma,1)$, $k\in\Omega$, which are algebraically independent such that
 \[\mathcal{A}^{\textit{sym}}\cong\C[f_{4},\dots,f_{42}]\,.\]
\end{cor}
This section is dedicated to the explicit construction of the $f_{k}$. 
The restriction map of Lemma \ref{lemma:Restriction of Maass space} extends to a homomorphism of the graded algebras
\[\begin{aligned}
   &\operatorname{Res}^{\mathcal{O}}_{\mathfrak{o}}\,:\,\mathcal{B}(\mathcal{O})\to\mathcal{B}(\mathfrak{o})\,,&
  \end{aligned}\]
where $\mathcal{B}(R)$ is the algebra generated by
\[\begin{aligned}
   &\bigoplus_{k=0}^{\infty}{\mathcal{M}(k,R)}\,,\,&&R\in\{\mathcal{O},\mathfrak{o}\}\,.&
  \end{aligned}\]

\begin{theorem}\label{theorem:symmetric generators for graded ring}
 The ring $\mathcal{A}^{\textit{sym}}$ is generated by the functions $E_{k}$, \(k\in \Omega\,.\) 
\end{theorem}
\begin{proof}
 From Theorem \ref{theorem:Shiga} we know that $\mathcal{A}^{\textit{sym}}$ is a weighted polynomial ring in 11 variables. In particular this ring is generated by the homogeneous components $\mathcal{A}^{\textit{sym}}_{k},k\in\Omega$. We consider the subalgebra over $\C$ which is generated by the functions stated above and denote by $\mathcal{E}_{k}$ its $k$-th homogeneous component. Moreover let
\[\mathcal{E}^{\ast}_{k}=\{\operatorname{Res}^{\mathcal{O}}_{\mathfrak{o}}(f)\,|\,f\in\mathcal{E}_{k}\}\quad,\quad E_{k}^{\ast}=\operatorname{Res}^{\mathcal{O}}_{\mathfrak{o}}(E_{k})\,.\]
Obviously $\mathcal{E}^{\ast}_{k}$ coincides with the $k$-th homogeneous component of $\C[E_{4}^{\ast},\dots,E_{42}^{\ast}]$. We can calculate the dimensions of these spaces using the Fourier expansion of $E_{k}^{\ast}$ in Lemma \ref{lemma:Restriction of Maass space}.  One has
\begin{center}
\begin{tabular}{l|lllllllllll}
$k$ & 4 & 10 & 12 & 16 & 18 & 22 & 24 & 28 & 30 & 36 & 42\\\hline
$\dim \mathcal{E}_{k}^{\ast}$ & 1 & 1 & 2 & 3 & 2 & 4 & 6 & 9 & 8 & 17 & 23\\
$\dim \mathcal{A}_{k}^{\textit{sym}}$ & 1 & 1 & 2 & 3 & 2 & 4 & 6 & 9 & 8 & 17 & 23
\end{tabular}
\end{center}
where the last row is extracted from Corollary \ref{cor:abstract ring of symmetric mf}. Because of \[\dim \mathcal{E}_{k}^{\ast}\leq \dim \mathcal{E}_{k}\leq \dim \mathcal{A}_{k}^{\textit{sym}}\] we are done.
\end{proof}
\begin{rem}
\begin{enumerate}[a)]
 \item The reason, why the restriction method works here, is due to the fact that there is no modular form of weight $6$ on the Cayley half-space.
If we consider the restriction to the Siegel upper half-plane instead, we obtain the following table for the dimensions
\begin{center}
\begin{tabular}{l|lllllllllll}
$k$ & 4 & 10 & 12 & 16 & 18 & 22 & 24 & 28 & 30 & 36 & 42\\\hline
$\dim$ & 1 & 1 & 2 & 3 & 2 & 4 & 6 & 9 & 8 & 17 & 22
\end{tabular}
\end{center}
Thus we need the restriction to the Hermitian half-space only for the weight

$k=42$.
\item  In view of Lemma \ref{lemma:Restriction of Maass space} it is necessary to consider all integral Cayley numbers up to a certain norm. For the computations with Siegel modular forms this bound is given by 12 and for Hermitian modular we have to deal with all numbers up to norm 20. 
Since the theta series associated with the integral Cayley numbers coincides with the elliptic Eisenstein series of weight $4$, we get
\[\sharp\{r\in\mathcal{O}\,|\,N(r)=n\}=240\,\sigma_{3}(n)\quad,\quad n\in \N\,.\]
The coefficients needed to compute the dimensions of the $\mathcal{E}_{k}$ can be found on the third author's homepage:
\begin{flushleft}
  \url{http://www.matha.rwth-aachen.de/~woitalla/publikationen/}.
\end{flushleft}
\end{enumerate}
\end{rem}
\section{The skew-symmetric modular form}
In order to construct a skew-symmetric modular form, we take a different approach which is also used in \cite{Gr} to construct reflective modular forms. Let $II_{2,26}$ be the unique (up to isomorphism) even unimodular lattice of signature $(2,26)$. We define 
\[\index{$R_{-2}^{II_{2,26}}$}R_{-2}^{II_{2,26}}:=\{r\in II_{2,26}\,|\,(r,r)=-2\}\,.\]
Let $\mathcal{D}(II_{2,26})$ be the homogeneous domain of type IV for $II_{2,26}$. For any $r\in R_{-2}^{II_{2,26}} $ we define the rational quadratic divisor as
\[\mathcal{D}_{r}(II_{2,26})=\{[\mathcal{Z}]\in\mathcal{D}(II_{2,26})\,|\,(\mathcal{Z},r)=0\}\,.\]
The following statement is due to Borcherds and can be found in \cite{Bo2}, Theorem 10.1 and Example 2.
\begin{theorem}[Borcherds]\label{theorem:construction of Borcherds Phi12}
 There is a holomorphic modular form $\Phi_{12}$ with the properties \index{$\Phi_{12}$}\[\Phi_{12}\in\mathcal{M}_{12}(\operatorname{O}(II_{2,26})^{+},\det)\quad, \quad \operatorname{div}(\Phi_{12}) =\bigcup_{r\in R^{II_{2,26}}_{-2}}{\mathcal{D}_{r}(II_{2,26})}\,,\] 
where the vanishing order is exactly one on each irreducible component.
\end{theorem}
In \cite{Bo2}, Example 2, Borcherds calculates the Fourier expansion of $\Phi_{12}$. It turns out that $\Phi_{12}$ reflects the Weyl denominator formula for the fake monster Lie algebra. 
\vspace{5mm}

\noindent We consider the embedding $L_{2}\hookrightarrow II_{2,26}$ and denote by $K$ the orthogonal complement of $L_{2}$ in $II_{2,26}$. Each vector $r\in II_{2,26}$ has a unique decomposition
\[r=\alpha(r)+\beta(r)\quad,\quad \alpha(r)\in L_{2}\,,\,\beta(r)\in K^{\sharp},\]
where 
\[K^{\sharp}=\{l \in K\otimes\R\,|\,(l,h)\in\Z\;\;\text{for all}\;\; h \in K\}\] 
is the dual lattice. We set
\[\index{$R_{-2}(K)$}R_{-2}(K)=\{r\in II_{2,26}\,|\,(r,r)=-2\,,\, r\perp L_{2} \},\]
which is contained in the negative definite lattice $K$ and hence finite. Consequently we define \index{$\operatorname{N}(K)$}$\operatorname{N}(K)=\frac{\sharp R_{-2}(K)}{2}\in\N\,.$
The next statement is a special case of Theorem 8.2 and Corollary 8.12 in \cite{GHS2} and describes the construction of a quasi-pullback from Borcherds function $\Phi_{12}$.
\begin{theorem}\label{theorem:quasi pullback form phi12}
Fix an embedding $L_{2}\hookrightarrow II_{2,26}$ and set
\[\left.\Phi\right|^{\textit{(QP)}}_{L_{2}}(\mathcal{Z})=\displaystyle\left.\frac{\Phi_{12}(\mathcal{Z})}{\prod_{r\in R_{-2}(K)/{\{\pm 1\}}}{(\mathcal{Z},r)}}\right|_{\mathcal{D}},\]
where in the product one fixes a set of representatives for $R_{-2}(K)/\{\pm 1\}$. Then  $\left.\Phi\right|^{\textit{(QP)}}_{L_{2}}$ belongs to $\mathcal{M}_{12+\operatorname{N}(K)}(\operatorname{O}(L_{2})^{+},\det)$
and vanishes exactly on all rational quadratic divisors  
\[\mathcal{D}_{\alpha(r)}=\{[\mathcal{Z}]\in\mathcal{D}\,|\,(\mathcal{Z},\alpha(r))=0\}\] 
where $r$ runs through the set \(R_{-2}^{II_{2,26}}\) and $(\alpha(r),\alpha(r))<0$. If $\operatorname{N}(K)>0$ we say that $\left.\Phi\right|^{\textit{(QP)}}_{L_{2}}$ is a quasi-pullback of $\Phi_{12}$\index{quasi-pullback of $\Phi_{12}$}. In this case $\left.\Phi\right|^{\textit{(QP)}}_{L_{2}}$ is a cusp form. 
 \end{theorem}
Let $\mathcal{N}$ be the Niemeier lattice with root system $E_{8}\perp E_{8}\perp E_{8}$, see \cite{Nie}. Since $II_{2,26}\cong U \perp U_{1}\perp \mathcal{N}$, where 
\[U,U_{1}\cong\begin{pmatrix}
0&1\\1&0                                                                                                                                                                                                                                                                                                                                                   \end{pmatrix}\] 
are two integral hyperbolic planes, we can consider the natural embedding \[L_{2}\hookrightarrow U \perp U_{1}\perp \mathcal{N}\,,\]  where $L$ is identified with one of the three copies of $L$ in the root system of $\mathcal{N}$. Denote by $\Phi$ the quasi-pullback of $\Phi_{12}$ which corresponds to this embedding.
\begin{cor}\label{cor:construction of skew-symmetric form}
 There is a skew-symmetric cusp form $\Phi$ of weight 252, whose divisor coincides with the $\Gamma_{2}$-orbit of
\[\left\{\begin{pmatrix}
            \omega &z\\\overline{z}&\tau
           \end{pmatrix}
\in\mathbb{H}_{2}(\mathcal{C})\,\Big|\,\overline{z}=-z\right\}\,.\]
The vanishing order of $\Phi$ is one on each irreducible component of $\operatorname{div}(\Phi)$.
\end{cor}
\begin{proof}
For the determination of the weight we use the well-known fact that the number of $2$-roots of $E_{8}$ is $240$. Hence $\operatorname{N}(K)=240$ and Theorem \ref{theorem:quasi pullback form phi12} yields that $\Phi$ is a skew-symmetric cusp form of weight $252$. Let $r=\alpha(r)+\beta(r)\in R_{-2}^{II_{2,26}}$ such that $(\alpha,\alpha)< 0$. The choice of our embedding already implies $(\alpha(r),\alpha(r))=-2$ and $\beta(r)=0$. From the last Theorem we deduce that each irreducible component of $\operatorname{div}(\Phi)$ is determined by a $-2$-root in $L_{2}$ and the vanishing order of $\Phi$  is one. Let $g\in\operatorname{SO}(L_{2})^{+}$ and let $\alpha\in L_{2}$ be a -2-root. One has
\[g.\mathcal{D}_{\alpha}=\mathcal{D}_{g^{-1}.\alpha}\]
 Since $L_{2}$ is unimodular, the Eichler criterion implies that $\operatorname{div}(\Phi)$ equals the $\operatorname{SO}(L_{2})^{+}$-orbit of $\mathcal{D}_{\alpha}$, see \cite{GHS}. We can choose the representative 
\begin{equation}\label{-2-representative for divisor}
 \alpha=(0,0,r,0,0)\in L_{2}\quad\text{ where }\quad r=\frac{e_{0}+e_{5}+e_{6}+e_{7}}{2}
\end{equation}
 to obtain the forementioned subset of $\mathbb{H}_{2}(\mathcal{C})$ as the biholomorphic image of $\mathcal{D}_{\alpha}$.
\end{proof}
 \begin{rem}
  This modular form has already been constructed in \cite{FS}, Lemma 5.1. One may also use \cite{Bo2}, Theorem 10.1, to construct $\Phi$ as a Borcherds lift. To this end we consider the nearly holomorphic modular form for the group $\sl$
\[\frac{g_{4}^{2}}{\Delta}=q^{-1}+504+16404q+q^{2}(\cdots)=\sum_{n\geq -1}{c(n)q^{n}}\,,\,q=\exp(2\pi i \tau),\tau\in\mathbb{H}\]
where $\Delta$ denotes the normalized Ramanujan $\Delta$-function. Now the multiplicative lifting yields a modular form of weight $\frac{c(0)}{2}=252$ whose zeroes lie on the rational quadratic divisors $\mathcal{D}_{r}$ where $r\in II_{2,10}$ such that $(r,r)=-2$. The multiplicity of each $D_{r}$ is $\sum_{n>0}c(-n^{2})=1$. Hence this modular form coincides with $\Phi$ up to a non-zero multiple in $\C$.
 \end{rem}
 This leads us to the main result.
\begin{theorem}\label{theorem:structure full graded ring}
 The graded ring $\mathcal{A}$ is generated by the $E_{k},k\in\Omega$, and $\Phi$. There is a unique polynomial $p\in\C[X_{1},\dots,X_{11}]$ such that 
\[\Phi^{2}=p(E_{4},E_{10},E_{12},E_{16},E_{18},E_{22},E_{24},E_{28},E_{30},E_{36},E_{42})\,.\]
\end{theorem}
\begin{proof}
 We show that any modular form can be expressed as a polynomial in the 12 functions from above. For convenience we represent modular forms as functions on the domain $\mathcal{H}(L_{2})$. Let $f\in\mathcal{A}$. In view of Theorem \ref{theorem:symmetric generators for graded ring} we can assume that $f$ is skew-symmetric. We take $\alpha\in L_{2}$ as in (\ref{-2-representative for divisor}) and define $M_{\alpha}\in\operatorname{O}(L_{2})^{+}$ as the reflection at the hyperplane perpendicular to $\alpha$.  From the identities
\[\begin{aligned}
&f(M_{\alpha}\langle Z \rangle)=\det(M_{\alpha})f(Z)=-f(Z)\,,&&&\\
&M_{\alpha}\langle Z \rangle=(\omega,-\overline{z},\tau)\,,&&&  \end{aligned}\]
which hold for all $ Z=(\omega,z,\tau)\in\mathcal{H}(L_{2})$ we conclude that $f$ vanishes on $\mathcal{D}_{\alpha}$. Using Theorem \ref{cor:construction of skew-symmetric form} we can divide $f$ by $\Phi$ and the quotient $f/\Phi$ is a symmetric modular form by Koecher's principle for automorphic forms, compare \cite{Ba}, p. 209. Now Theorem \ref{theorem:symmetric generators for graded ring} implies that $f$ is contained in the algebra generated by $\Phi$ and $\mathcal{A}^{\textit{sym}}$. The square of $\Phi$ is a symmetric form. Since $E_{4},\dots,E_{42}$ are algebraically independent we infer that there exists a polynomial $p$ with the desired property.
\end{proof}
The algebraic structure of $\mathcal{A}$ has already been determined in \cite{HU} where the authors refined the isomorphism given in Theorem \ref{theorem:Shiga}  to a bimeromorphic map of orbifolds. 
\vspace{5mm}

\noindent Following \cite{Kl} we can perform a different construction for $\Phi$ by considering a Rankin-Cohen type differential operator. 
\begin{prop}\label{prop:Wronski-Determinant}
 Let $f_{j}\in\mathcal{M}_{k_{j}}(\Gamma,1)$, $j=1,\dots 11$. Then $\{f_{1},\dots,f_{11}\}\in\mathcal{M}_{k}(\Gamma,\det)$, where
\[ \{f_{1},\dots,f_{11}\}:=\det\begin{pmatrix}
                              k_{1}f_{1}&\dots &k_{11}f_{11}\\[0.5em]
			      \frac{\partial f_{1}}{\partial \omega}&\dots &\frac{\partial f_{11}}{\partial \omega}\\[0.5em]
			      \frac{\partial f_{1}}{\partial z_{1}}&\dots &\frac{\partial f_{11}}{\partial z_{1}}\\
			      \vdots & \vdots & \vdots\\
			      \frac{\partial f_{1}}{\partial z_{8}}&\dots &\frac{\partial f_{11}}{\partial z_{8}}\\[0.5em]
			      \frac{\partial f_{1}}{\partial \tau}&\dots &\frac{\partial f_{11}}{\partial \tau}
                             \end{pmatrix}\quad,\quad k=10+\sum_{j=1}^{11}{k_{j}}\,.\]
If $f_{1},\dots,f_{11}$are algebraically independent, then $\{f_{1},\dots,f_{11}\}$ does not vanish identically.
\end{prop}
\begin{proof}
For a proof we refer to \cite{Kl}, Proposition 2.1, and \cite{AI}, Proposition 2.1.
\end{proof}
We state an immediate consequence of Theorem \ref{theorem:symmetric generators for graded ring} and Theorem \ref{theorem:structure full graded ring}.
\begin{cor}
 There is a constant $0\neq c\in \C$ such that \[\Phi=c\,\{E_{4},\dots,E_{42}\}\,.\]
\end{cor}
\vspace{2ex}
\textbf{Acknowledgement:} The authors thank H. Hashimoto and T. Ueda for helpful discussions and in particular for pointing out the proof of Corollary 4.2 from \cite{HU}.
\vspace{2ex}

\bibliography{bibliography_woitalla} 
\bibliographystyle{plain}

\end{document}